\documentclass[12pt]{amsart}
\usepackage{amssymb}
\usepackage{fullpage}
\usepackage[colorlinks]{hyperref}
\usepackage{dsfont}
\usepackage{eucal}
\usepackage[all]{xy}
\SelectTips{eu}{}

\newcommand{\ep}{\varepsilon}

\newcommand{\from}{\leftarrow}

\newcommand{\phat}{{}^{^\wedge}_p}

\newcommand{\Tor}{\mathsf{Tor}}

\newcommand{\bZ}{\mathbb Z}

\renewcommand{\le}{\leqslant}

\newcommand{\tensor}{\otimes}
\newcommand{\eps}{\epsilon}

\numberwithin{equation}{section}

\theoremstyle{plain}
\newtheorem{lemma}{Lemma}[section]
\newtheorem{theorem}[lemma]{Theorem}
\newtheorem{proposition}[lemma]{Proposition}
\newtheorem{corollary}[lemma]{Corollary}

\theoremstyle{definition}

\newtheorem{hypothesis}[lemma]{Hypothesis}

\theoremstyle{remark} 
\newtheorem{remark}[lemma]{Remark} 
\newtheorem{remarks}[lemma]{Remarks}

\author{Dave Benson} 
\author{John Greenlees}
\title{Massey products in the homology of the loopspace of a
  $p$-completed classifying space:  
finite groups with cyclic Sylow $p$-subgroups}
\begin{document}

\begin{abstract}
Let $G$ be a finite group with cyclic Sylow $p$-subgroup, 
and let $k$ be a field of characteristic $p$. Then
$H^*(BG;k)$ and $H_*(\Omega BG\phat;k)$ are $A_\infty$ algebras whose
structure we determine up to quasi-isomorphism.
\end{abstract}

\maketitle

\section{Introduction}
The general context is that we have a finite group $G$, and a field
$k$ of characteristic $p$. We are interested in the differential
graded cochain algebra $C^*(BG;k)$
and the differential graded algebra $C_*(\Omega (BG\phat);k)$ of chains
on the loop space: these two are Koszul dual to each 
other, and the Eilenberg-Moore and Rothenberg-Steenrod spectral
sequences relate the cohomology ring $H^*(BG;k)$  to the homology 
ring $H_*(\Omega (BG\phat);k)$. Of course if $G$ is a $p$-group, $BG$ is $p$-complete so $\Omega 
(BG\phat)\simeq G$, but in general $H_*(\Omega (BG\phat); k)$ is 
infinite dimensional. Henceforth we will omit the brackets from 
$\Omega (BG\phat)$. \\[1ex]

We consider a simple case where
the two rings are not formal, but we can identify the $A_{\infty}$
structures precisely (see Section \ref{sec:Ainfty} for a brief summary on $A_{\infty}$-algebras). From now on we suppose specifically that  $G$ is a finite group with cyclic Sylow $p$-subgroup $P$, 
and let $BG$ be its classifying space. Then the inclusion of
the Sylow $p$-normaliser $N_G(P) \to G$ and the quotient map $N_G(P)
\to N_G(P)/O_{p'}N_G(P)$
induce mod $p$ cohomology 
equivalences 
\[ B(N_G(P)/O_{p'}N_G(P)) \from BN_G(P) \to BG, \] 
and
hence homotopy equivalences after
$p$-completion 
\[ B(N_G(P)/O_{p'}N_G(P))\phat \xleftarrow{\ \sim\ } BN_G(P)\phat
  \xrightarrow{\ \sim \ } BG\phat. \]
Here, $O_{p'}N_G(P)$ denotes the largest normal $p'$-subgroup of
$N_G(P)$. Thus $N_G(P)/O_{p'}N_G(P)$ is a semidirect product 
$\bZ/p^n\rtimes \bZ/q$, where $q$ is a divisor of $p-1$, and 
$\bZ/q$ acts faithfully as a group of automorphisms of $\bZ/p^n$. In
particular, the isomorphism type of $N_G(P)/O_{p'}N_G(P)$ only depends
on $|P|=p^n$ and the inertial index $q=|N_G(P):C_G(P)|$, and therefore
so does the homotopy type of $BG\phat$. Our main theorem determines
the multiplication maps $m_i$ in the $A_\infty$ structure on 
$H^*(BG;k)$ and $H_*(\Omega (BG\phat);k)$ arising from $C^*(BG;k)$ and
$C_*(\Omega (BG\phat); k)$ respectively. We will suppose from now on
that $p^n>2, q>1$ since the case of a $p$-group is well understood. 

The starting point is the cohomology ring
$$H^*(BG; k)=H^*(B\bZ/p^n; k)^{\bZ/q}=k[x]\otimes \Lambda(t)
\mbox{ with }|x|=-2q, |t|=-2q+1.  $$
There is a preferred generator $t\in
H^1(B\bZ/p^n;k)=\mathrm{Hom}(\bZ/p^n,k)$ and we take $x$ to be the
$n$th Bockstein of $t$.

Before stating our result we should be clear about grading and signs.  

\begin{remark}
\label{rem:deg}
We will be discussing both homology and cohomology, so we should be
explicit that everything is graded homologically, so that
differentials always lower degree. Explicitly, the degree of an
element of $H^i(G;k)$ is $-i$.
\end{remark}

\begin{remark}
\label{rem:signs}
Sign conventions for Massey products and $A_{\infty}$ algebras mean
that a specific sign will enter repeatedly in our statements, so for
brevity we  write 
$$\eps (s)=
\begin{cases}
+1 & s\equiv 0, 1 \mod 4\\
-1 &s\equiv 2, 3 \mod 4\\
\end{cases}. $$
\end{remark}

\begin{theorem}
Let $G$ be a finite group with cyclic Sylow $p$-subgroup $P$ of order 
$p^n$ and inertial index $q$ so that 
$$H^*(BG;k)
=k[x]\otimes \Lambda(t) \mbox{ with } |x|=-2q, |t|=-2q+1 \mbox{ and } \beta_nt=x$$  Up to quasi-isomorphism, the $A_\infty$ structure on $H^*(BG;k)$ is determined by 
\[ m_{p^n}(t,\dots,t)=\eps (p^n) x^{h} \]  
where $h=p^n-(p^n-1)/q$. This implies 
\[ m_{p^n}(x^{j_1}t,\dots,x^{j_{p^n}}t)=\eps (p^n) x^{h+j_1+\dots+j_{p^n}}\]
for all $j_1, \ldots, j_{p^n}\geq 0$.  
 All $m_i$ for $i>2$ on all other
$i$-tuples of monomials give zero.

If $q>1$ and  $p^n\ne 3$ then 
\[ H_*(\Omega BG\phat;k) = k[\tau] \otimes \Lambda(\xi) \mbox{ where }
  |\tau|=2q-2, |\xi|=2q-1. \]
 Up to quasi-isomorphism, the
$A_\infty$ structure is determined  by
\[ m_h(\xi,\dots,\xi )=\eps (h) \tau^{p^n}.  \]
This implies 
\[ m_h(\tau^{j_1}\xi,\dots,\tau^{j_h}\xi)=\eps (h) \tau^{p^n+j_1+\dots+j_h} \]
for all $j_1, \ldots, j_{p^n}\geq 0$.   All $m_i$ for $i>2$ on all other $i$-tuples of monomials give
zero.

If $q>1$ and $p^n=3$ then $q=2$ and  
\[ H_*(\Omega BG\phat;k) = k[\tau,\xi]/(\xi^2+\tau^3), \]
and all $m_i$ are zero for $i>2$. 
\end{theorem}

\section{The group algebra and its cohomology}

We assume from now on, without loss of generality, 
that $G$ has a normal cyclic Sylow $p$-subgroup
$P=C_G(P)$, with inertial index $q=|G:P|$. We shall assume that $q>1$,
which then forces $p$ to be odd. For notation, let
\[ G=\langle g,s\mid g^{p^n}=1,\ s^q=1,\ 
sgs^{-1}=h^\gamma\rangle\cong \bZ/p^n\rtimes\bZ/q, \]
where $\gamma$ is a primitive $q$th root of unity modulo $p^n$.
Let $P=\langle g\rangle$ and $H=\langle s\rangle$ as subgroups of
$G$. 

Let $k$ be a field of characteristic $p$. 
The action of $H$ on
$kP$ by conjugation preserves the radical series, and since $|H|$ is
not divisible by $p$, there are invariant complements. Thus we may
choose an element $U\in J(kP)$ such that $U$ spans an $H$-invariant
complement of $J^2(kP)$ in $J(kP)$. It can be checked that
\[ U = \sum_{\substack{1\le j \le p^n-1, \\(p,j)=1}} g^j/j \]
is such an element, and that $sUs^{-1}=\gamma U$. This gives us the
following presentation for $kG$:
\[ kG = k\langle s,U \mid U^{p^n}=0,\ s^q=1,\ sU = \gamma Us
  \rangle. \]

We shall regard $kG$ as a $\bZ[\frac{1}{q}]$-graded 
algebra with $|s|=0$ and $|U|=1/q$. Then the bar resolution
is doubly graded, and taking homomorphisms into $k$, the
cochains $C^*(BG;k)$ inherit a double grading. The differential
decreases the homological grading and preserves  the internal
grading. Thus the cohomology $H^*(G,k)=H^*(BG;k)$ is doubly graded:
\[ H^*(BG;k) = k[x] \otimes \Lambda(t) \]
where $|x|=(-2q,p^n)$, $|t|=(-2q+1,h)$, and $h=p^n-(p^n-1)/q$.
Here, the first degree is homological, the second internal.
The Massey product $\langle t,t,\dots,t\rangle$ ($p^n$ repetitions) is
equal to $-x^h$. This may easily be determined by restriction to $P$,
where it is well known that the $p^n$-fold Massey product of the
degree one exterior generator is a non-zero degree two element.
The usual convention is to make the constant $-1$, because
this Massey product is minus the $n$th Bockstein of $t$ \cite[Theorem 14]{Kraines:1966a}.

\section{$A_\infty$-algebras}
\label{sec:Ainfty}

An $A_{\infty}$-algebra over a field is a $\bZ$-graded vector space
$A$ with graded maps $m_n: A^{\otimes n}\rightarrow A$ of degree $n-2$ for
$n\geq 1$ satisfying 
$$\sum_{r+s+t=n}(-1)^{r+st}m_{r+1+t}(id^{\tensor r}\tensor m_s \tensor
id^{\tensor t})=0 $$
for $n\geq 1$. The map $m_1$ is therefore a differential, and the map
$m_2$ induces a product on $H_*(A)$.

A theorem of Kadeishvili~\cite{Kadeishvili:1982a} (see also
Keller~\cite{Keller:2001a,Keller:2002a} or
Merkulov~\cite{Merkulov:1999a}) may be stated as
follows. Suppose that we are given a differential
graded algebra $A$,  over a field $k$. Let $Z^*(A)$ be the cocycles,
$B^*(A)$ be the coboundaries, and $H^*(A)=Z^*(A)/B^*(A)$. 
Choose a vector space splitting $f_1\colon H^*(A) \to Z^*(A) \subseteq A$
of the quotient.
Then this gives by an inductive procedure an $A_\infty$ structure
on $H^*(A)$ so that the map $f_1$ is the degree one part of a quasi-isomorphism of
$A_\infty$-algebras. 

If $A$ happens to carry auxiliary gradings respected by the product
structure and preserved by the differential, then it is easy to check
from the inductive procedure that
the maps in the construction may be chosen so that they also
respect these gradings. It then follows that the structure maps $m_i$
of the $A_\infty$ structure on $H^*(A)$ also respect these gradings.

Let us apply this to $H^*(BG;k)$. We examine the elements $m_i(t,\dots,t)$.
By definition, we have $m_1(t)=0$ and $m_2(t,t)=0$. The degree of 
$m_i(t,\dots,t)$ is $i$ times the degree of $t$, increased in the
homological direction by $i-2$. This gives
\[ |m_i(t,\dots,t)| = i(-2q+1,h)+(i-2,0) =(-2iq+2i-2,ih). \]
The homological degree is even, so if $m_i(t, \cdots , t)$ is non-zero then it is a
multiple of a power of $x$. Comparing degrees, if $m_i(t,\dots,t)$ is
a non-zero multiple of $x^\alpha$ then we have 
\[ 2iq-2i+2=2\alpha q,\qquad ih = \alpha p^n. \] 
Eliminating $\alpha$, we obtain $(iq-i+1)p^n=ihq$. Substituting
$h=p^n-(p^n-1)/q$, this gives $i=p^n$. Finally, since the Massey
product of $p^n$ copies of $t$ is equal to $-x^h$, it follows that
$m_{p^n}(t,\dots,t)=\eps (p^n) x^h$, where the sign is as defined in Remark
\ref{rem:signs} \cite[Theorem 3.1]{LuPalmieriWuZhang:2009}. Thus we have
\[ m_i(t,\dots,t) = \begin{cases} \eps (p^n) x^h & i=p^n \\ 
0 & \text{otherwise.} \end{cases} \]
We shall elaborate on this argument in a more general context in the
next section, where we shall see that the rest of the $A_\infty$
structure is also determined in a similar way.

\section{$A_\infty$ structures on a polynomial tensor exterior algebra}

In this section, we shall examine the following general situation.
Our goal is to establish that there are only two possible $A_\infty$
structures satisfying Hypothesis \ref{hyp:grading} below, and that the Koszul dual also
satisfies the same hypothesis with the roles of $a$ and $b$, and of
$h$ and $\ell$ reversed.

\begin{hypothesis}
\label{hyp:grading}
$A$ is a $\bZ\times\bZ$-graded $A_\infty$-algebra over a field $k$,
where the operators $m_i$ have degree $(i-2,0)$, satisfying
\begin{enumerate}
\item $m_1=0$, so that $m_2$ is strictly associative, 
\item ignoring the $m_i$ with $i>2$, the algebra $A$ is
$k[x] \otimes \Lambda(t)$ where $|x|=(-2a,\ell)$ and $|t|=(-2b-1,h)$, and
\item $ha-\ell b = 1$.
\end{enumerate}
\end{hypothesis}

\begin{remarks}
(i) The $A_\infty$-algebra $H^*(BG;k)$ of the last section satisfies this
hypothesis, with $a=q$, $b=q-1$, $h=p^n-(p^n-1)/q$, $\ell=p^n$.

(ii) By comparing degrees, if we have $m_\ell(t,\dots,t)=\eps (p^n) x^h$ then
$(2b+1)\ell + 2-\ell = 2ah$ and so $ha-\ell b = 1$. This explains the
role of part (3) of the hypothesis. The consequence is, of course,
that $a$ and $b$ are coprime, and so are $h$ and $\ell$.
\end{remarks}

\begin{lemma}
If $m_i(t,\dots,t)$ is non-zero, then $i=\ell > 2$ and $m_\ell(t,\dots,t)$ is
a multiple of $x^h$.
\end{lemma}
\begin{proof}
The argument is the same as in the last section. The degree of
$m_i(t,\dots,t)$ is $i|t| +i- 2  = (-2ib - 2, ih)$. Since the
homological degree is even, if $m_i(t,\dots,t)$ is non-zero then it is
a multiple of some power of $x$, say $x^\alpha$. Then we have
\[ 2ib+2 = 2\alpha a,\qquad ih = \alpha\ell. \]
Eliminating $\alpha$ gives $(ib+1)\ell=iha$, and so 
using $ha-\ell b =1$ we have $i=\ell$. Substituting back gives 
$\alpha = h$.
\end{proof}

Elaborating on this argument gives the entire $A_\infty$ structure.
If $m_\ell(t,\dots,t)$ is non-zero, then by rescaling the variables
$t$ and $x$ if necessary
we can assume that $m_\ell(t,\dots,t)= x^h$ (note that we can even do
this without extending the field, since $\ell$ and $h$ are coprime).

\begin{proposition}
If $m_\ell(t,\dots,t)=0$ then all $m_i$ are zero for $i>2$. If
$m_\ell(t,\dots,t)= x^h$ then
$m_\ell(x^{j_1}t,\dots,x^{j_\ell}t)= x^{h+j_1+\dots+j_\ell}$, 
and all $m_i$ for $i>2$ on all other $i$-tuples of monomials give zero.
\end{proposition}
\begin{proof}
All monomials live in different degrees, so we do not need to consider
linear combinations of monomials.
Suppose that
$m_i(x^{j_1}t^{\ep_1},\dots,x^{j_i}t^{\ep_i})$ is some constant multiple
of $x^jt^\ep$, 
where each of $\ep_1,\dots,\ep_i,\ep$ is either zero or one. Then
comparing degrees, we have
\[ (j_1+\dots+j_i)|x| + (\ep_1+\dots+\ep_i)|t| + (i-2,0) = j|x| +  \ep|t|. \]
Setting
\[ \alpha=j_1+\dots+j_i-j,\qquad \beta = \ep_1+\dots+\ep_i-\ep \]
we have $\beta \le i$, and
\[ \alpha(-2a,\ell) + \beta(-2b-1,h) + (i-2,0) = 0. \]
Thus
\[ 2\alpha a + 2\beta b + \beta + 2 - i = 0, \qquad \alpha \ell +
  \beta h = 0. \]
Eliminating $\alpha$, we obtain
\[ -2\beta ha + 2\beta \ell b + \beta \ell + 2 \ell - i \ell = 0. \]
Since $ha-\ell b=1$, this gives $\beta = \ell (i-2) /(\ell -2)$.
Combining this with $\beta \le i$ gives $i\le \ell$. If $i<\ell$ then
$\beta$ is not divisible by $\ell$, and so $\alpha \ell + \beta h = 0$
cannot hold. So we have $\beta=i=\ell$, $\ep_1=\dots=\ep_\ell=1$,
$\ep=0$, $\alpha=-h$, and $j=h+j_1+\dots+j_\ell$.
Finally, the identities satisfied by the $m_i$ for an
$A_\infty$ structure show that all the constant multiples have to be
the same, hence all equal to zero or after rescaling, all equal to minus one.
\end{proof}

Combining the lemma with the proposition, we obtain the following.

\begin{theorem}\label{th:Ainfty}
Under the hypothesis above, if $\ell > 2$ then there are two possible $A_\infty$
structures on $A$. 
There is the formal one, where $m_i$ equal to zero for $i>2$, 
and the non-formal one, where  
after replacing $x$ and $t$ by suitable multiples, the only
non-zero $m_i$ with $i>2$ is $m_\ell$, and the only non-zero values on
monomials are given by
\begin{equation*} 
m_\ell(x^{j_1}t,\dots,x^{j_\ell}t)=x^{h+j_1+\dots+j_\ell}. 
\end{equation*}
\end{theorem}

\begin{theorem}
Let $G=\bZ/p^n \rtimes \bZ/q$ as above, and $k$ a field of
characteristic $p$. Then
the $A_\infty$ structure on $H^*(G,k)$ given by Kadeishvili's theorem
may be taken to be the non-formal possibility named in the above
theorem, with $a=q$, $b=q-1$, $h=p^n-(p^n-1)/q$, $\ell=p^n$.
\end{theorem}
\begin{proof}
Since we have $m_{p^n}(t,\dots,t)=\eps (p^n) x^h$, the formal possibility does
not hold.
\end{proof}

\begin{remark}
Dag Madsen's thesis~\cite{Madsen:2002a} 
has an appendix in which the $A_\infty$ structure
is computed for the cohomology of a truncated polynomial ring, reaching similar
conclusions by more direct methods.
\end{remark}

\section{Loops on $BG\phat$}

In general for a finite group $G$ we have $H^*(BG\phat;k)\cong
H^*(BG;k)=H^*(G,k)$ and $\pi_1(BG\phat)=G/O^p(G)$,
the largest $p$-quotient of $G$. In our case, $G=\bZ/p^n\rtimes\bZ/q$
with $q>1$, we have $G=O^p(G)$ and so $BG\phat$ is simply
connected. So the Eilenberg--Moore spectral sequence converges to the
homology of its loop space:
\[ \Tor_{*,*}^{H^*(G,k)}(k,k) \Rightarrow H_*(\Omega BG\phat;k). \]
The internal grading on $C^*(BG;k)$ gives this spectral sequence a
third grading that is preserved by the differentials, and $H_*(\Omega
BG\phat;k)$ is again doubly graded. Since $H^*(G,k)=k[x] \otimes
\Lambda(t)$ with $|x|=(-2q,p^n)$ and $|t|=(-2q+1,h)$, it follows that
the $E^2$ page of this spectral sequence is
$k[\tau] \otimes \Lambda(\xi)$
where $|\xi|=(-1,2q,p^n)$ and $|\tau|=(-1,2q-1,h)$ 
(recall $h=p^n-(p^n-1)/q$). Provided that we are not in the case
$h=2$, which only happens if 
$p^n=3$, ungrading $E^\infty$ gives
\[ H_*(\Omega BG\phat;k) = k[\tau] \otimes \Lambda(\xi) \]
with $|\tau|=(2q-2,h)$ and $|\xi|=(2q-1,p^n)$. 

In the exceptional case $h=2$, $p^n=3$, we have $q=2$, and
the group $G$ is the symmetric group $\Sigma_3$ of degree
three. An explicit computation (for example by squeezed resolutions 
\cite{Benson:2009b}) gives
\[ H_*(\Omega (B\Sigma_3){}^{^\wedge}_3;k) = k[\tau,\xi]/(\xi^2+\tau^3) \]
with $|\tau|=(2,2)$ and $|\xi|=(3,3)$, and the two gradings collapse
to a single grading. 

Applying
Theorem~\ref{th:Ainfty}, and using the fact that either formal case is
Koszul dual to the other, we have the following.

\begin{theorem}
Suppose that $p^n\ne 3$. Then
the $A_\infty$ structure on $H_*(\Omega
BG\phat;k)=k[\tau]\otimes\Lambda(\xi)$ is given by
\[ m_h(\tau^{j_1}\xi,\dots,\tau^{j_h}\xi)=\eps (h) \tau^{p^n+j_1+\dots+j_h}, \]
and for $i>2$, all $m_i$ on all other $i$-tuples of monomials give zero.\qed
\end{theorem}

Using \cite{LuPalmieriWuZhang:2009} again, we
may reinterpret this in terms of Massey products. 

\begin{corollary}
In $H_*(\Omega BG\phat)$,
the Massey products $\langle \xi,\dots,\xi \rangle$ ($i$ times) vanish
for $0<i<h$, and give $-\tau^{p^n}$ for $i=h$.\qed
\end{corollary}

Note that the exceptional case $p^n=3$ also fits the corollary, if we
interpret a $2$-fold Massey product as an ordinary product.


\newcommand{\noopsort}[1]{}
\providecommand{\bysame}{\leavevmode\hbox to3em{\hrulefill}\thinspace}
\providecommand{\MR}{\relax\ifhmode\unskip\space\fi MR }
\providecommand{\MRhref}[2]{%
  \href{http://www.ams.org/mathscinet-getitem?mr=#1}{#2}
}
\providecommand{\href}[2]{#2}

\end{document}